\documentclass[11pt]{article}
\usepackage[latin1]{inputenc}
\usepackage{amsmath,amsthm,amssymb}
\usepackage[dvips]{graphicx}
\usepackage[small]{caption}
\usepackage{indentfirst}
\usepackage[section]{placeins}
\usepackage{graphicx}
\usepackage{color}
\usepackage{floatflt,graphicx}

\newtheorem{thm}{Theorem}
\numberwithin{thm}{section}%
\newtheorem{lemma}[thm]{Lemma}
\newtheorem{cor}[thm]{Corollary}
\newtheorem{prop}[thm]{Proposition}
\newtheorem{rem}[thm]{Remark}

\begin{document}

\title{Asymptotic Analysis of Boundary Layer \\
Correctors and Applications}

\author{Daniel Onofrei and Bogdan Vernescu}

\maketitle

\abstract {In this paper we extend the ideas presented in Onofrei
and Vernescu [\textit{Asymptotic Analysis, 54, 2007, 103-123}] and
introduce suitable second order boundary layer correctors, to
study the $H^1$-norm error estimate for the classical problem in
homogenization. Previous second order boundary layer results
assume either smooth enough coefficients (which is equivalent to
assuming smooth enough correctors $\chi_j,\chi_{ij}\in
W^{1,\infty}$), or smooth homogenized solution $u_0$, to obtain an
estimate of order $\displaystyle O(\epsilon^{\frac{3}{2}})$. For
this we use the periodic unfolding method developed by Cioranescu,
Damlamian and Griso [\textit{C. R. Acad. Sci. Paris, Ser. I 335,
2002, 99-104}]. We prove that in two dimensions, for nonsmooth
coefficients and general data,  one obtains an estimate of order
$\displaystyle O(\epsilon^\frac{3}{2})$. In three dimenssions the
same estimate is obtained assuming $\chi_j,\chi_{ij}\in W^{1,p}$,
with $p>3$. We also discuss how our results extend, in the case of
nonsmooth coefficients, the convergence proof for the finite
element multiscale method proposed by T.Hou et al. [\textit{ J. of
Comp. Phys., 134, 1997, 169-189}] and the first order corrector
analysis for the first eigenvalue of a composite media obtained by
Vogelius et al.[\textit{Proc. Royal Soc. Edinburgh, 127A, 1997,
1263-1299}].}





\newcommand{\abs}[1]{\lvert#1\rvert}

\section{Introduction}

This paper is dedicated to the study of error estimates for the classical problem in homogenization using suitable boundary layer correctors.

 \noindent Let $\Omega\in{\mathbb R}^N$, denote a bounded convex polyhedron
 or  a convex bounded domain with a sufficiently smooth boundary. Consider also the unit cube $Y=(0,1)^N$.
  It is well known that for $A\in L^\infty(Y)^{N\times N}$,
 $Y$-periodic with $\displaystyle m|\xi|^2\leq A_{ij}(y)\xi_i\xi_j\leq M|\xi|^2$, for any $\xi\in{{\mathbb R}}^N$, the solutions of
\begin{equation}
\label{I-epsilon-problem}
 \left\{\begin{array}{ll}
 -\nabla\cdot(A(\displaystyle\frac{x}{\epsilon})\nabla u_\epsilon(x))=f
 & \mbox{ in }\Omega \\
 u_\epsilon = 0 & \mbox{ on }\partial\Omega
 \end{array}\right .\end{equation}
have the property that (see \cite{SA}, \cite{JKO},
\cite{BP},\cite{LPB}),
$$u_\epsilon\rightharpoonup u_0\;\mbox{ in } H_0^1(\Omega)$$
where $u_0$ verifies
\begin{equation}
\label{I-hom-problem} \left\{\begin{array}{ll}
 -\nabla\cdot ({\cal{A}}^{hom}\nabla u_0(x))=f
 & \mbox{ in }\Omega \\
 u_0 = 0 & \mbox{ on }\partial\Omega
 \end{array}\right .\end{equation}
with
\begin{equation}
\label{I-hom-coeff}{\cal{A}}^{hom}_{ij}=\displaystyle
M_Y\left(A_{ij}(y)+A_{ik}(y)\frac{\partial\chi_j}{\partial y_k}\right)\end{equation}
where
$M_Y(\cdot)=\displaystyle\frac{1}{|Y|}\int_Y\cdot dy$ and $\chi_j\in
W_{per}(Y)=\{\chi\in H^1_{per}(Y)|M_Y(\chi)=0\}$ are the solutions of the local problem
\begin{equation}
\label{I-first-local-1}
 -\nabla_y\cdot(A(y)(\nabla\chi_j+e_j))=0
\end{equation}
Here $e_j$ represent the canonical basis in ${\mathbb R}^N$.

 \noindent In this paper, $\nabla$ and $(\nabla
\cdot)$ will denote the full gradient and divergence operators
respectively, and with $\nabla_x,(\nabla_x\cdot)$ and
$\nabla_y,(\nabla_y\cdot)$ will denote the gradient and the
divergence in the slow and fast variable respectively.
\begin{rem}
\label{I-Stein-ext} Throughout this paper, we will denote
by $\Phi$ the continuous extension of a given function $\Phi\in
W^{p,m}(\Omega)$ with $p,m\in{\mathbb Z}$, to the space
$W^{p,m}({\mathbb R} ^N)$. With minimal assumption on the smoothness
of $\Omega$ a stable extension operator can be constructed
(see \cite{ST}, Ch. VI, 3.1).
\end{rem}
The formal asymptotic expansion corresponding to the above results
can be written as
$$ u_\epsilon(x)=u_0(x)+\epsilon \displaystyle
w_1(x,\frac{x}{\epsilon})+...$$ where
\begin{equation}
\label{I-first-cor} \displaystyle
w_1(x,\frac{x}{\epsilon})=\displaystyle\chi_j(\frac{x}{\epsilon})\frac{\partial
u_0}{\partial x_j}
\end{equation}
We make the observation that the Einstein summation convention
will be used and that the letter $C$ will denote a constant
independent of any other parameter, unless otherwise specified.

 \noindent A classical result (see \cite{SA}, \cite{JKO}, \cite{L},\cite{BP}), states that with additional regularity assumptions on the local problem solutions $\chi_j$ or on $u_0$ one has
\begin{equation}
\label{I-1.1} \displaystyle||u_\epsilon(\cdot)-u_0(\cdot)-\epsilon
w_1(\cdot,\frac{.}{\epsilon})||_{H^1(\Omega)}\leq
C\epsilon^{\frac{1}{2}}
\end{equation}
Without any additional assumptions a similar result has been
recently proved by G. Griso in \cite{G}, using the Periodic
Unfolding method developed in \cite{CDG}, i.e.,
\begin{equation}
\label{I-1.2} ||u_\epsilon(\cdot)-u_0(\cdot)-\epsilon
\chi_j(\frac{.}{\epsilon})Q_\epsilon(\frac{\partial u_0}{\partial
x_j})||_{H^1(\Omega)}\leq
C\epsilon^{\frac{1}{2}}||u_0||_{H^2(\Omega)}
\end{equation}
with
$$ x\in {\tilde{\Omega}}_\epsilon,\;
\displaystyle
\;Q_\epsilon(\phi)(x)=\sum_{i_1,..,i_N}M_Y^\epsilon(\phi)(\epsilon\xi+\epsilon
i){\bar x}_{1,\xi}^{i_1}\cdot...{\bar
x}_{N,\xi}^{i_N},\;\;\;\xi=\displaystyle\left[\frac{x}{\epsilon}\right
]$$ for $\phi\in L^2(\Omega)$, $i=(i_1,...,i_N)\in\{0,1\}^N$ and
$$
{\bar x}_{k,\xi}^{i_k}=\left\{\begin{array}{ll}
\displaystyle\frac{x_k-\epsilon\xi_k}{\epsilon} & \mbox{if } i_k=1\\
1-\displaystyle\frac{x_k-\epsilon\xi_k}{\epsilon} & \mbox{ if }
i_k=0\end{array}\right. \;\;x\in \epsilon(\xi+Y)$$ where
$M_Y^\epsilon(\phi)=\displaystyle\frac{1}{\epsilon^N}\int_{\epsilon\xi+\epsilon
Y}\phi(y)dy$ and
${\tilde{\Omega}}_\epsilon=\displaystyle\bigcup_{\xi\in{\mathbb Z}^N}\left\{\epsilon\xi+\epsilon
Y; (\epsilon\xi+\epsilon Y)\cap\Omega
\neq\emptyset\right\}.$

 \noindent In order to improve the error estimates in (\ref{I-1.1}) boundary
layer terms have been introduced as solutions to

\begin{equation}
\label{I-first-ord-bl}
\displaystyle-\nabla\cdot(A(\frac{x}{\epsilon})\nabla
\theta_\epsilon)=0\;\mbox{ in
}\;\Omega\;\;,\;\;\theta_\epsilon=w_1(x,\frac{x}{\epsilon})\;\mbox{
on }\;\partial\Omega
\end{equation}
Assuming $A\in C^\infty(Y)$, $Y$-periodic matrix and a sufficiently
smooth homogenized solution $u_0$ it has been proved in \cite{LPB}
(see also \cite{L}) that
\begin{equation}
\label{I-1.3} \displaystyle||u_\epsilon(\cdot)-u_0(\cdot)-\epsilon
w_1(\cdot,\frac{.}{\epsilon})+\epsilon\theta_\epsilon(\cdot)||_{H_0^1(\Omega)}\leq
C\epsilon
\end{equation}
\begin{equation}
\label{I-1.4} \displaystyle||u_\epsilon(\cdot)-u_0(\cdot)-\epsilon
w_1(\cdot,\frac{.}{\epsilon})+\epsilon\theta_\epsilon(\cdot)||_{L^2(\Omega)}\leq
C\epsilon^2.
\end{equation}
In \cite{MV}, Moskow and Vogelius proved the above estimates
assuming $A\in C^\infty(Y)$, $Y$-periodic matrix and $u_0\in
H^2(\Omega)$ or $u_0\in H^3(\Omega)$ for (\ref{I-1.3}) or
(\ref{I-1.4}) respectively. Inequality (\ref{I-1.3}) is proved in
\cite{AA} for the case when $A\in L^{\infty}(Y)$ and $u_0\in
W^{2,\infty}(\Omega)$.

 \noindent In \cite{SV}, Sarkis and Versieux showed that the estimates
(\ref{I-1.3}) and respectively (\ref{I-1.4})  still holds in a
more general setting, when one has $u_0\in W^{2,p}(\Omega)$, $
\chi_j\in W^{1,q}_{per}(Y)$ for (\ref{I-1.3}), and $u_0\in
W^{3,p}(\Omega)$, $ \chi_j\in W^{1,q}_{per}(Y)$ for (\ref{I-1.4}),
where, in both cases, $p>N$ and $q>N$ satisfy
$\displaystyle\frac{1}{p}+\frac{1}{q}\leq\frac{1}{2}$. In
\cite{SV} the constants in the right hand side of (\ref{I-1.3})
and (\ref{I-1.4}) are proportional to $\displaystyle
||u_0||_{W^{2,p}(\Omega)}$ and $\displaystyle
||u_0||_{W^{3,p}(\Omega)}$ respectively.

 \noindent In order to improve the error estimate in (\ref{I-1.3}) and
(\ref{I-1.4}) one needs to consider the second order boundary
layer corrector, $\varphi_\epsilon$ defined as the solution of,
\begin{equation}
\label{I-second-ord-bl}
\displaystyle-\nabla\cdot(A(\frac{x}{\epsilon})\nabla
\varphi_\epsilon)=0\;\mbox{ in
}\;\Omega\;\;,\;\;\varphi_\epsilon(x)=\displaystyle\chi_{ij}
(\frac{x}{\epsilon})\frac{\partial^2 u_0}{\partial x_i\partial
x_j}\;\mbox{ on }\;\partial\Omega
\end{equation}
where $\chi_{ij}\in W_{per}(Y)$ are solution of the following local
problems,
\begin{equation}
\label{I-second-local}
\nabla_y\cdot(A\nabla_y\chi_{ij})=b_{ij}+{\cal{A}}^{hom}_{ij}
\end{equation}
with ${\cal{A}}^{hom}$ defined by (\ref{I-hom-problem}),
$M_Y(b_{ij}(y))=-{\cal{A}}^{hom}_{ij}$, and $\displaystyle
b_{ij}=-A_{ij}-A_{ik}\frac{\partial \chi_j}{\partial
y_k}-\frac{\partial}{\partial y_k}(A_{ik}\chi_j)$.

 \noindent For the case when $u_0\in W^{3,\infty}(\Omega)$ and $ \chi_{ij}\in
W^{1,\infty}(Y)$, with the help of $\varphi_\epsilon$  defined in
(\ref{I-second-ord-bl}), Alaire and Amar proved in \cite{AA} the
following result
\begin{eqnarray}
\label{I-1.7*}
 \displaystyle||u_\epsilon(\cdot)-u_0(\cdot)-\epsilon
w_1(\cdot,\frac{.}{\epsilon})+\epsilon
\theta_\epsilon(\cdot)-\epsilon^2\displaystyle\chi_{ij}
(\frac{\cdot}{\epsilon})\frac{\partial^2 u_0}{\partial x_i\partial
x_j}||_{H^1(\Omega)}& \leq & \nonumber\\
&&\nonumber\\
\leq C\epsilon^{\frac{3}{2}}||u_0||_{W^{3,\infty}(\Omega)}&&
\end{eqnarray}
This result shows that with the help of the second order
correctors one can essentially improve the order of the estimate
(\ref{I-1.3}). In the general case of nonsmooth periodic
coefficients, $A\in L^\infty(Y)$, and $u_0\in H^2(\Omega)$,
inspired by Griso's idea, we proved in \cite{OV}
\begin{equation}
\label{I-1.5'} \displaystyle ||u_\epsilon(\cdot)-u_0(\cdot)-\epsilon
\chi_j(\frac{\cdot}{\epsilon})Q_\epsilon(\frac{\partial
u_0}{\partial
x_j})+\epsilon\beta_\epsilon(\cdot)||_{H_0^1(\Omega)}\leq
C\epsilon||u_0||_{H^2(\Omega)}
\end{equation}
with $\beta_\epsilon$ defined by
\begin{equation}
\label{I-first-ord-bl-reg}
\displaystyle-\nabla\cdot(A(\frac{x}{\epsilon})\nabla
\beta_\epsilon)=0\;\mbox{ in
}\;\Omega\;\;,\;\;\beta_\epsilon=u_1(x,\frac{x}{\epsilon})\;\mbox{
on }\;\partial\Omega
\end{equation}
where $\displaystyle u_1(x,\frac{x}{\epsilon})\doteq
\chi_j(\frac{x}{\epsilon})Q_\epsilon(\frac{\partial u_0}{\partial
x_j})$.

 \noindent When $u_0\in W^{3,p}(\Omega)$ with $p>N$ we also proved in
\cite{OV} that
\begin{equation}
\label{I-1.6'} \displaystyle||u_\epsilon(\cdot)-u_0(\cdot)-\epsilon
\chi_j(\frac{\cdot}{\epsilon})\frac{\partial u_0}{\partial
x_j}+\epsilon\theta_\epsilon(\cdot)||_{L^2(\Omega)}\leq
C\epsilon^2||u_0||_{W^{3,p}(\Omega)}.
\end{equation}
In this paper, we present a refinement of (\ref{I-1.7*}) for the
case of nonsmooth coefficients and general data. To do this we
start by describing the asymptotic behavior of $\varphi_\epsilon$
with respect to $\epsilon$.

  \noindent The key difference between the case of smooth coefficients, and the nonsmooth case
discussed in the present paper is that in the former, by means of
the maximum principle or Avellaneda's compactness results (see
\cite{A}), it can be proved that the second order boundary layer
corrector $\varphi_\epsilon$ is bounded in $L^2(\Omega)$ and is of
order $\displaystyle O(\frac{1}{\sqrt{\epsilon}})$ in
$H^1(\Omega)$, while in the latter one cannot use the
aforementioned techniques to describe the asymptotic behavior of
$\varphi_\epsilon$ in $L^2({\Omega})$ or $H^1(\Omega)$. Moreover
one can see that $\varphi_\epsilon$ is not bounded in
$L^2(\Omega)$ in general ( see \cite{A}), and therefore one needs
to carefully address the question of the asymptotic behavior of
$\varphi_\epsilon$ with respect to $\epsilon$.

  \noindent First, we can easily observe that $\epsilon\varphi_\epsilon$ can be interpreted as the
 solution of an elliptic problem with variable periodic
 coefficients and with weakly convergent data in $H^{-1}(\Omega)$.
 For this class of problems a result of Tartar, \cite{T}(see also \cite{DC})
 implies
 $$\displaystyle\epsilon\varphi_\epsilon\stackrel{\epsilon}{\rightharpoonup}0\;\mbox{
 in }\;H^1(\Omega)$$
 As a consequence of Lemma \ref{I-lemma-5.2} we obtain that for
$u_0\in H^3(\Omega)$ and $\chi_j,\chi_{ij}\in W_{per}^{1,p}(Y)$,
for some $p>N$, we have
\begin{equation}
\label{I-1.8}
\displaystyle||\epsilon\varphi_\epsilon||_{H^1(\Omega)}\leq
C\epsilon^{\frac{1}{2}}||u_0||_{H^3(\Omega)}
\end{equation}
Using (\ref{I-1.8}) we are able to prove that for $u_0\in
 H^3(\Omega)$ and $\chi_j,\chi_{ij}\in W_{per}^{1,p}$ with $p>N$ we have
\begin{equation}
\label{I-1.8'} \displaystyle ||u^\epsilon(\cdot)-u_0(\cdot)-\epsilon
\chi_j(\frac{.}{\epsilon})\frac{\partial u_0}{\partial
x_j}+\epsilon\theta_\epsilon(\cdot)-\epsilon^2\displaystyle\chi_{ij}
(\frac{.}{\epsilon})\frac{\partial^2 u_0}{\partial x_i\partial
x_j}||_{H^1(\Omega)}\leq C\displaystyle\epsilon^{\frac{3}{2}}||u_0||_{H^3(\Omega)}.
\end{equation}
Remark \ref{I-Meyers} states that in two dimensions due to a Meyer
type regularity for the solutions of the cell problems,
$\chi_j,\chi_{ij}$, estimate (\ref{I-1.8'}) holds only assuming
$u_0\in H^3(\Omega)$.

 \noindent In Section 2.4 we use (\ref{I-1.8'}) to extend the results in
\cite{MV} to the case of nonsmooth coefficients. Namely, in two
dimensions Moskow and Vogelius (see \cite{MV}) considered the
Dirichlet spectral problem associated to (\ref{I-epsilon-problem})
\begin{equation}
\label{I-spectral-epsilon}
 \left\{\begin{array}{ll}
 -\nabla\cdot(A(\displaystyle\frac{x}{\epsilon})\nabla
 u_\epsilon(x))=\lambda^\epsilon u_\epsilon
 & \mbox{ in }\Omega \\
 u_\epsilon = 0 & \mbox{ on }\partial\Omega
 \end{array}\right .\end{equation}
The eigenvalues of (\ref{I-spectral-epsilon}) form an increasing
sequence of positive numbers, i.e,
$$0<\lambda_1^\epsilon\leq\lambda_2^\epsilon\leq...\leq\lambda_j^\epsilon\leq...$$
and it is well known that we have $\lambda_j^\epsilon\rightharpoonup
\lambda_j$ as $\epsilon\rightarrow 0$ for any $j\geq 0$ where
$$0<\lambda_1\leq\lambda_2\leq...\leq\lambda_j\leq...$$
are the Dirichlet eigenvalues of the homogenized operator, i.e.,
\begin{equation}
\label{I-spectral-hom}
 \left\{\begin{array}{ll}
 -\nabla\cdot({\cal{A}}^{hom}\nabla
 u(x))=\lambda u
 & \mbox{ in }\Omega \\
 u = 0 & \mbox{ on }\partial\Omega
 \end{array}\right .\end{equation}
For $A\in C^\infty(Y)$, $Y$-periodic, and assuming that the
eigenfunctions of (\ref{I-spectral-hom}) belong to
$H^{2+r}(\Omega)$, with $r>0$, Moskow and Vogelius analysed in
\cite{MV}, the first corrector of the homogenized eigenvalue of
(\ref{I-spectral-hom}) and proved that (See Thm. 3.6), up to a
subsequence,
\begin{equation}
\label{I-1.9}
\displaystyle\frac{\lambda^\epsilon-\lambda}{\epsilon}\rightarrow\lambda\int_\Omega\theta_*udx
\end{equation}
where $\theta_*$ is a weak limit of $\theta_\epsilon$ in
$L^2(\Omega)$, and $u$ is the normal eigenvector associated to the
eigenvalue $\lambda$.

 \noindent Using (\ref{I-1.8'}) we show that the result obtained in \cite{MV}
for the first corrector of the homogenized eigenvalue holds true
in the general case of nonsmooth periodic coefficients $A\in
L^{\infty}_{per}(Y)$.

\section{A Fundamental Result}

In this section we analyze the asymptotic behavior with respect to
$\epsilon$ of the solutions to a certain class of elliptic
problems with highly oscillating coefficients and boundary data.
The main result is stated in Proposition \ref{I-lemma-5.2} but we
will first present a technical Lemma which will be useful in what
follows,

\begin{lemma}
\label{local-av-ineq} Let $\Phi$ be such that $\Phi\in
W^{1,p}_{per}(Y)$ with $p>N$, and let $\psi\in H^1(\Omega)$. Then
we have
\begin{equation}
\displaystyle\int_{\Omega}\vert\nabla_y\Phi({x\over\epsilon})\vert^2(\psi(x)-M_Y^\epsilon(\psi)(x))^2dx\leq
C\epsilon^2\vert\vert\Phi\vert\vert^2_{W^{1,p}(Y)}\vert\vert\psi\vert\vert^2_{H^1(\Omega)}
\end{equation}
\end{lemma}
\begin{proof}
Let ${\tilde\Omega}_\epsilon\subset {\mathbb R}^N$ be the smallest
union of integer translates of $\epsilon Y$ that cover $\Omega$,
i.e.
$${\tilde\Omega}_\epsilon\doteq\bigcup_{\xi\in
Z_\epsilon}(\xi\epsilon+\epsilon Y)$$ where
$$Z_\epsilon\doteq\{\xi\in{\mathbb Z}^N,\; (\xi\epsilon+\epsilon
Y)\cap\Omega\neq \emptyset\}$$ We start by recalling that there
exists a linear and continuous extension operator
${\cal{P}}:H^1(\Omega)\rightarrow H^1({\tilde\Omega}_\epsilon)$,
with the continuity constant independent of $\epsilon$ (see
\cite{G,G1} for details). In the rest of this section, without
having to specify it every time, every function in $H^(\Omega)$
will be extended trough $\cal{P}$ to
$H^1({\tilde\Omega}_\epsilon)$. Next we proceed with the proof of
the Lemma. We have
\begin{eqnarray}
\label{I-2.0}
&\displaystyle\int_{\Omega}\vert\nabla_y\Phi({x\over\epsilon})\vert^2(\psi(x)-M_Y^\epsilon(\psi)(x))^2dx
\leq
\displaystyle\int_{{\tilde\Omega}_\epsilon}\vert\nabla_y\Phi({x\over\epsilon})\vert^2(\psi(x)-M_Y^\epsilon(\psi)(x))^2dx\nonumber\\
&\displaystyle\leq \sum_{\xi\in
Z_\epsilon}\int_{\xi\epsilon+\epsilon
Y}\vert\nabla_y\Phi({x\over\epsilon})\vert^2(\psi(x)-M_Y^\epsilon(\psi)(x))^2dx\nonumber\\
&\leq\displaystyle\sum_{\xi\in
Z_\epsilon}\epsilon^N\int_Y\vert\nabla_y\Phi\vert^2(\psi(\xi\epsilon+\epsilon
y)-M_Y^\epsilon(\psi)(\xi\epsilon+\epsilon y))^2dy
\end{eqnarray}
Let $\psi(\xi\epsilon+\epsilon y)=z_\xi(y)$. Using this in
(\ref{I-2.0}) we obtain,
\begin{eqnarray}
\label{I-2.1}& \displaystyle\int_{\Omega}\vert\nabla_y
\Phi({x\over\epsilon})\vert^2(\psi(x)-M_Y^\epsilon(\psi)(x))^2dx
\leq\nonumber\\
&&\nonumber\\
&\leq \displaystyle\sum_{\xi\in
Z_\epsilon}\epsilon^N\int_Y\vert\nabla_y\Phi\vert^2\left(z_\xi(y)-{1\over
|Y|}\int_Y z_\xi(s)ds\right)^2dy\nonumber\leq\\
&&\nonumber\\
 &\leq\displaystyle \sum_{\xi\in
Z_\epsilon}\epsilon^N \Vert\Phi\Vert^2_{W^{1,p}(Y)}\;\left\Vert
z_\xi-{1\over|Y|}\int_Y
z_\xi(s)ds\right\Vert^2_{L^{{2p\over{p-2}}}(Y)}\end{eqnarray}
 Note that $\nabla_y z_\xi
= \epsilon \nabla_x\psi(\xi\epsilon+\epsilon y)$. Then, using this
 one can easily observe that (\ref{I-app5}) in the Appendix,
together with the Poincare-Wirtinger inequality, implies
\begin{equation}\label{I-2.2}
\left\Vert z_\xi-{1\over|Y|}\int_Y
z_\xi(s)ds\right\Vert_{L^{{2p\over{p-2}}}(Y)}\leq c_p\Vert\nabla_y
z_\xi\Vert_{L^2(Y)}\end{equation}
From (\ref{I-2.2}) in
(\ref{I-2.1}) we have
\begin{eqnarray}
&\displaystyle\int_{\Omega}\vert\nabla_y\Phi({x\over\epsilon})\vert^2(\psi(x)-M_Y^\epsilon(\psi)(x))^2dx\leq\nonumber\\
&\leq c_p\epsilon^N \Vert\Phi\Vert^2_{W^{1,p}(Y)}\displaystyle
\sum_{\xi\in Z_\epsilon}\Vert\nabla_y
z_\xi\Vert^2_{L^2(Y)}\nonumber\\&=c_p\epsilon^{N+2}
\Vert\Phi\Vert^2_{W^{1,p}(Y)}\displaystyle \sum_{\xi\in
Z_\epsilon}\displaystyle\int_Y
\left(\nabla_x\psi(\xi\epsilon+\epsilon y)\right)^2dy=\nonumber\\
&=c_p\epsilon^{2} \Vert\Phi\Vert^2_{W^{1,p}(Y)}\displaystyle
\sum_{\xi\in Z_\epsilon}\int_{\epsilon\xi+\epsilon
Y}|\nabla_x\psi|^2dx\leq\nonumber\\
&&\nonumber\\
 &\leq C\epsilon^{2}
\Vert\Phi\Vert^2_{W^{1,p}(Y)}\Vert\psi\Vert^2_{H^1(\Omega)}\end{eqnarray}
where $C$ depends on $p$ only. So the statement of the Lemma is
proved.
\end{proof}

\begin{prop}
\label{I-lemma-5.2} Let $\Omega\subset {\mathbb R}^N$ be bounded
and either of class $C^{1,1}$ or convex. Consider the following
problem,
\begin{equation}
\label{I-star-5.2} \left\{
\begin{array}{ll}-\nabla\cdot(A(\frac{x}{\epsilon})\nabla
y_\epsilon)=h & \;\mbox{ in }\;\Omega\\
y_\epsilon=g_\epsilon & \;\mbox{ on
}\;\partial\Omega\end{array}\right .\end{equation}
where $h\in L^2(\Omega)$, the coefficient matrix $A$ satisfies the
hypothesis of the first section, and we have that there exists
$\phi_*\in W^{1,p}_{per}(Y)$ with $p>N$, and $z_\epsilon$ a bounded sequence in
$H^1(\Omega)$ such that
\begin{equation}
\label{I-square'}
g_\epsilon(x)=\epsilon\phi_*(\frac{x}{\epsilon})z_\epsilon(x)\;\mbox{
a.e. }\; \Omega.
\end{equation}
Then  there exists $y_*\in H_0^1(\Omega)$ such that
\begin{equation}
\label{I-5.2} y_\epsilon\rightharpoonup y_*\;\mbox{ in
}\;H^1(\Omega)
\end{equation}
and $y_*$ satisfies
\begin{equation}
\label{I-5.3} \left\{\begin{array}{ll}
\nabla\cdot({\cal{A}}^{hom}\nabla y_*)=h
&\;\mbox{ in }\;\Omega\\
y_*=0 &\;\mbox{ on }\;\partial\Omega\end{array}\right.\end{equation}
and ${\cal{A}}^{hom}$ is the classical homogenized matrix defined in
(\ref{I-hom-coeff}). Moreover we have
\begin{equation}
\label{I-5.5}
\displaystyle||y_\epsilon-y_*-\epsilon\chi_j(\frac{x}{\epsilon})Q_\epsilon(\frac{\partial
y_*}{\partial x_j})||_{H^1(\Omega)}\leq
C\displaystyle\epsilon^\frac{1}{2}\left(1+||y_*||_{H^2(\Omega)}\right)
\end{equation}
where  $\chi_j\in W_{per}(Y)$ are
defined in (\ref{I-first-local-1}), $Q_\epsilon$ is defined in
(\ref{I-1.2}) and $C$ depends only on $p$.
\end{prop}
\begin{proof}
To prove (\ref{I-5.2}) and (\ref{I-5.3}) Tartar's result
concerning problems with weakly converging data in $H^{-1}$ could
be used. We prefer to present here a different proof based on the
periodic unfolding method developed in \cite{CDG}, which will also imply (\ref{I-5.5}). First observe
that the solution of (\ref{I-star-5.2}) satisfy,
\begin{equation}
\label{I-5.6}
y_\epsilon=y_\epsilon^{(1)}+y_\epsilon^{(2)}+y_\epsilon^{(3)}\end{equation}
where $y_\epsilon^{(1)}, y_\epsilon^{(2)}, y_\epsilon^{(3)}$
satisfy respectively,
\begin{equation}
\label{I-5.7}\left\{\begin{array}{ll}
&-\nabla\cdot(A({x\over\epsilon})\nabla
y_\epsilon^{(1)})=h\; \mbox{
in }\Omega ,\\
&y_\epsilon^{(1)}=0\; \mbox{ on
}\partial\Omega ,\end{array}\right.\end{equation}
\begin{equation}
\label{I-5.8}\left\{\begin{array}{ll}
&-\nabla\cdot(A({x\over\epsilon})\nabla
y_\epsilon^{(2)})=0\; \mbox{
in }\Omega ,\\
&y_\epsilon^{(2)}=\epsilon\Phi_*({x\over\epsilon})Q_\epsilon(z_\epsilon)\; \mbox{
on }\partial\Omega ,\end{array}\right.\end{equation}
\begin{equation}
\label{I-5.9}\left\{\begin{array}{ll}
&-\nabla\cdot(A({x\over\epsilon})\nabla
y_\epsilon^{(3)})=0\; \mbox{
in }\Omega ,\\
&y_\epsilon^{(3)}=\epsilon\Phi_*({x\over\epsilon})(z_\epsilon-Q_\epsilon(z_\epsilon))\;
\mbox{ on }\partial\Omega .\end{array}\right.\end{equation} First
note that from Theorem 4.1 in \cite{G}, stated here in
(\ref{I-1.2}), we have
\begin{equation}
\label{I-5.10} \Vert y^{(1)}_\epsilon(x)-y_*(x)-\epsilon
\chi_j(\frac{x}{\epsilon})Q_\epsilon(\frac{\partial y_*}{\partial
x_j})\Vert_{H^1(\Omega)}\leq C\epsilon^{\frac{1}{2}}\Vert
y_*\Vert_{H^2(\Omega)}.
\end{equation}
From \cite{G} (see the two estimates before Theorem 4.1 there), by
using an interpolation inequality, we immediately arrive at,
 \begin{equation}
\label{I-5.11} \Vert y^{(2)}_\epsilon\Vert_{H^1(\Omega)}\leq
\displaystyle C\epsilon^{1\over 2}\Vert\Phi_*\Vert_{H^1(Y)}\;\Vert
z_\epsilon\Vert_{L^2(\Omega)}
\end{equation}
Finally, for $y^{(3)}_\epsilon$ we obtain,
\begin{eqnarray}
\label{I-5.12}\!\!\!\!\!\!\!\!\!&\Vert
y^{(3)}_\epsilon\Vert_{H^1(\Omega)}\leq C\Vert
\epsilon\Phi_*({x\over\epsilon})(z_\epsilon-Q_\epsilon(z_\epsilon))\Vert_{H^1(\Omega)}=\nonumber\\
&\nonumber\\
 &=C\Vert
\epsilon\Phi_*({x\over\epsilon})(z_\epsilon-Q_\epsilon(z_\epsilon))\Vert_{L^2(\Omega)}+C\Vert
\nabla_y\Phi_*({x\over\epsilon})(z_\epsilon-Q_\epsilon(z_\epsilon))\Vert_{L^2(\Omega)}+\nonumber\\
&\nonumber\\
&+C\Vert
\epsilon\Phi_*({x\over\epsilon})\nabla_x\left(z_\epsilon-Q_\epsilon(z_\epsilon)\right)\Vert_{L^2(\Omega)}
\leq \epsilon^2\Vert z_\epsilon \Vert_{H^1(\Omega)}
\Vert \Phi_*\Vert_{W^{1,p}(Y)} +\nonumber\\
&\nonumber\\
&+ C\Vert
\nabla_y\Phi_*({.\over\epsilon})\left(z_\epsilon-Q_\epsilon(z_\epsilon)\right)\Vert_{L^2(\Omega)}+\epsilon
\Vert \Phi_*\Vert_{W^{1,p}(Y)}\Vert z_\epsilon
\Vert_{H^1(\Omega)}\leq\nonumber\\
&\nonumber\\
&\leq C\Vert
\nabla_y\Phi_*({.\over\epsilon})(z_\epsilon-M_Y^\epsilon
z_\epsilon)\Vert_{L^2(\Omega)}+C\Vert
\nabla_y\Phi_*({.\over\epsilon})(Q_\epsilon
z_\epsilon-M_Y^\epsilon z_\epsilon)\Vert_{L^2(\Omega)}+
\nonumber\\
&\nonumber\\
 &+C\epsilon
\Vert \Phi_*\Vert_{W^{1,p}(Y)}\Vert z_\epsilon
\Vert_{H^1(\Omega)}\leq C\epsilon \Vert \Phi_*\Vert_{W^{1,p}(Y)}\Vert
z_\epsilon \Vert_{H^1(\Omega)}\end{eqnarray}
 where $C$ depends only on $p$ and where we used triangle
 inequality in the fourth line above and we used
 Lemma \ref{local-av-ineq} and $(\ref{I-app1}_5)$ of the Appendix, respectively, to
 estimate the first and the second terms in the fifth line.
From (\ref{I-5.10}), (\ref{I-5.11}), (\ref{I-5.12}) in
(\ref{I-5.6}) we obtain the statement of the Proposition.
\end{proof}

\section{Boundary layer error estimates}

In this section, for the case of $L^\infty$ coefficients, with the
only assumptions that $\chi_j,\chi_{ij}\in W_{per}^{1,p}(Y)$ for
some $p>N$ and $u_0\in H^3(\Omega)$ we show that the left hand
side of (\ref{I-1.7*}) is of order
$\displaystyle\epsilon^{\frac{3}{2}}$. Indeed we have,
\begin{thm}
\label{I-t3.2} Let $A\in L^\infty(Y)$ and $u_0\in H^3(\Omega)$. If
there exists $p>N$ such that $\chi_j,\chi_{ij}\in W_{per}^{1,p}(Y)$
 then we have
\begin{eqnarray}\left\|u_\epsilon(.)-u_0(.)-\epsilon
w_1(.,\frac{.}{\epsilon})+\epsilon\theta_\epsilon(.)
-\epsilon^2\chi_{ij}(\frac{\cdot}{\epsilon})\frac{\partial^2u_0}{\partial
x_i\partial x_j}\right\|_{H^1(\Omega)}\leq\nonumber\\
\leq C\epsilon^{\frac{3}{2}}||u_0||_{H^3(\Omega)}.\nonumber
\end{eqnarray}
\end{thm}
\begin{proof}
As we did before, for the sake of simplicity, we will assume $N=3$
the two dimensional case being similar. We will also assume for the
moment that the coefficients are smooth enough, as in Appendix,
Section \ref{smoothing}, relations
(\ref{I-app2}). For any
$i,j\in\{1,2,3\}$ let $\chi_{ij}^n\in W_{per}(Y)$ be the solutions of
\begin{equation}
\label{I-3.2} \nabla_y\cdot(A^n\nabla_y
\chi_{ij}^n)=b^n_{ij}-M_Y(b^n_{ij})
\end{equation}
where
$$\displaystyle b^n_{ij}=-A^n_{ij}-A^n_{ik}\frac{\partial\chi_j^n}{\partial
y_k}-\frac{\partial}{\partial y_k}(A^n_{ik}\chi_j^n)$$ and $M_Y(.)$
is the average on $Y$. From Corollary \ref{I-app-r4}
$$ |\nabla_y\chi_{ij}^n|_{L^2(Y)}<C\;\mbox{ and }
\;\chi_{ij}^n\rightharpoonup \chi_{ij}\;\mbox{ in
}W_{per}(Y),\;\;\forall i,j\in\{1,...,N\}$$ where
$$\displaystyle\int_Y A(y)\nabla_y\chi_{ij}\nabla_y\psi
dy=(b_{ij}-M_Y(b_{ij}),\psi)_{(W_{per}(Y),(W_{per}(Y))')}$$ for any
$\psi\in W_{per}(Y)$ and with
$$b_{ij}=-A_{ij}-A_{ik}\frac{\partial\chi_j}{\partial
y_k}-\frac{\partial}{\partial y_k}(A_{ik}\chi_j).$$ We define
$$
\displaystyle u^n_2(x,y)=\chi_{ij}^n(y)\frac{\partial ^2
u_0}{\partial x_j\partial x_i}(x)$$
and
\begin{equation}
\label{I-3.4'}
 \displaystyle(v_*^n(x,y))_k=A_{ki}^n(y)\chi_j^n(y)\frac{\partial
^2 u_0}{\partial x_j\partial
x_i}(x)+A_{kl}^n(y)\frac{\partial\chi_{ij}^n}{\partial
y_l}\frac{\partial ^2 u_0}{\partial x_j\partial x_i}
\end{equation}
 Following the same ideas as in \cite{MV} we can show that $\nabla_x\cdot
M_Y(v_*^n)=0$. Let
$$\displaystyle
R_{ki}^j=M_Y(A_{ki}^n\chi_j^n+A_{kl}^n\frac{\partial\chi_{ij}^n}{\partial
y_l}).$$
$$ (C^n(y))_{ij}=A_{ij}^n(y)+A_{ik}^n(y)\frac{\partial
\chi_j^n}{\partial y_k}$$
$${\cal{A}}^{hom}_n=M_Y(C^n(y))$$
Consider $\alpha_{ij}^n\in[L^2(Y)]^3$ defined by,
\begin{equation}
\nonumber \alpha_{ij}^n= \left(\begin{array}{lll} \vspace{0.2cm}
A_{1i}^n\chi_j^n+A_{1l}^n\frac{\partial\chi_{ij}^n}{\partial
y_l}-R_{1i}^j & \\
\vspace{0.2cm}
A_{2i}^n\chi_j^n+A_{2l}^n\frac{\partial\chi_{ij}^n}{\partial
y_l}-R_{2i}^j & \\
\vspace{0.2cm}
A_{3i}^n\chi_j^n+A_{3l}^n\frac{\partial\chi_{ij}^n}{\partial
y_l}-R_{3i}^j\end{array}\right)+\beta_{ij}^n
\end{equation}
with
\begin{equation}
\nonumber
\begin{array}{lll}
\vspace{0.2cm}
 \beta_{1j}^n=(0, -\phi_{3j}^n, \phi_{2j}^n)^T & \\
 \vspace{0.2cm}
 \beta_{2j}^n=(\phi_{3j}^n, 0, -\phi_{1j}^n)^T & \mbox{ for }
 j\in\{1,2,3\}\\
 \vspace{0.2cm}
 \beta_{3j}^n=(-\phi_{2j}^n, \phi_{1j}^n, 0)^T & \end{array}
 \end{equation}
where $T$ denotes the transpose.
The functions $\phi_{ij}^n\in W_{per}(Y)$
were defined in \cite{OV}, as solutions of
\begin{equation}
\label{I-2.8} curl_y\phi^n_l=B_l^n\;\mbox{ and }div_y\phi^n_l=0;
\end{equation}
where
$B^n(y)=C^n(y)-{\cal{A}}^{hom}_n$ and $B^n_l$ denotes the vector
$\displaystyle B^n_l=(B^n_{il})_{i}\in [L^2_{per}(Y)]^N$.
It was observed in \cite{OV} that for every $1\in\{1,2,...,N\}$
\begin{equation}\label{I-2.8'}
\displaystyle\phi_l^n\rightharpoonup\phi_l\;\mbox{ in }[W_{per}(Y)]^N \; \mbox{ where } \;
curl_y\phi_l=B_l\;\mbox{ and }div_y\phi_l=0
\end{equation}
The conditions on $\chi_j$, $\chi_{ij}$ and Remark 3.11 in \cite{GR} imply that $||\phi_l||_{W^{1,p}(Y)}<C$. Next, using the symmetry of the
matrix $A$ we observe that the vectors $\alpha_{ij}^n$ defined
above, are divergence free with zero average over $Y$. This implies
that there exists $\psi_{ij}^n\in \left[W_{per}(Y)\right]^3$, (see
Theorem 3.4, \cite{GR} adapted for the periodic case) so that
\begin{equation}
\label{I-3.5} curl_y\psi_{ij}^n=\alpha_{ij}^n\;\mbox{ and } \;
div\psi_{ij}^n=0\;\mbox{ for any }i,j\in\{1,2,3\}
\end{equation}
By using (\ref{I-2.8'}), Corollary \ref{I-app-r2}, Corollary
\ref{I-app-r4} in the definition of $\alpha^n_{ij}$ above, we
have,
\begin{equation}
\label{I-patrat}
 \alpha_{ij}^n\rightharpoonup \alpha_{ij}\;\mbox{ in
} [L^2(Y)]^3
\end{equation}
 where the form of $\alpha_{ij}$ is identical with that
of $\alpha_{ij}^n$ and can be obviously obtained from
(\ref{I-patrat}). Using the above convergence result and Theorem 3.9
from \cite{GR} adapted to the periodic case, we obtain that
$$\psi_{ij}^n\rightharpoonup \psi_{ij}\;,
\mbox { in }\; W_{per}(Y)\;\mbox{ for any } i,j\in\{1,2,3\}$$ and
$\psi_{ij}$ satisfy
\begin{equation}
\label{I-3.5'} curl_y\psi_{ij}=\alpha_{ij}\mbox{ and
}div_y\psi_{ij}=0\;\mbox{ for }i,j\in\{1,2,3\}
\end{equation}
The hypothesis on $\chi_j$ and $\chi_{ij}$ implies that
$\alpha_{ij}$ defined at (\ref{I-patrat}) belongs to the space
$[L^p(Y)]^3$ and for all pairs $(i,j)$ with $i,j\in\{1,2,3\}$ we
have
\begin{equation}
\label{I-4.1} \displaystyle||\alpha_{ij}||_{[L^p(Y)]^3}\leq
C(||\beta_{ij}||_{[L^p(Y)]^3}+||\chi_j||_{L^p(Y)}+||\chi_{ij}||_{W^{1,p}(Y)})\leq
C\
\end{equation}
Inequality (\ref{I-4.1}) and Remark 3.11 in \cite{GR}  imply that
\begin{equation}
\label{I-4star}
 \displaystyle||\psi_{ij}||_{[W^{1,p}(Y)]^3}\leq C\;\mbox{ for }i,j\in\{1,2,3\}
 \end{equation}
Define $\displaystyle p(x,y)=\psi_{ij}(y)\frac{\partial ^2
u_0}{\partial x_i\partial x_j}(x)$ and $v_2(x,y)=curl_x p(x,y)$.
We can see that $p\in H^1(\Omega, H^1_{per}(Y))$ and $v_2\in
L^2(\Omega, H^1_{per}(Y))$. Obviously we have that $\nabla_x\cdot
v_2=0$ in the sense of distributions (see \cite{MV}). Next, using
(\ref{I-3.4'}) we observe that $\displaystyle\nabla_x\cdot
M_{Y}(v_*)=0$ where $v_*$ is such that
$$\displaystyle v_*^n\rightharpoonup v_*\;\mbox{ weakly in }
L^2(\Omega,L^2_{per}(Y))$$ We have that
\begin{equation}
\nonumber
 \displaystyle(v_*(x,y))_k=A_{ki}(y)\chi_j(y)\frac{\partial
^2 u_0}{\partial x_j\partial
x_i}(x)+A_{kl}(y)\frac{\partial\chi_{ij}}{\partial
y_l}\frac{\partial ^2 u_0}{\partial x_j\partial x_i}
\end{equation}
Using this and the fact that
$$ \displaystyle \int_{\Omega\times Y}(\nabla_y\cdot v_2)\Phi(x,y)dxdy=
\displaystyle\int_{\Omega\times Y}(\nabla_y\cdot curl_x
p(x,y))\Phi(x,y)dxdy=$$
$$=\displaystyle -\int_{\Omega\times
Y}(\nabla_x\cdot curl_y p(x,y))\Phi(x,y)dxdy$$ for any smooth
function $\Phi\in{\cal{D}}(\Omega;{\cal{D}}(Y))$, one can
immediately see that
\begin{equation}
 \label{I-4.1''}\nabla_y\cdot
v_2=-\nabla_x\cdot v_*
\end{equation}
in the sense of distributions. Let $\displaystyle p^n(x,y)=\psi_{ij}^n(y)\frac{\partial ^2
u_0}{\partial x_i\partial x_j}(x)$ and $v_2^n(x,y)=curl_x p^n(x,y)$. Consider $\psi_\epsilon^n$ and
$\xi_\epsilon^n$ defined as follows,
 \begin{equation}
 \nonumber \begin{array}{cc} \displaystyle
w_1^n(x,y)=\chi_j^n(y)\frac{\partial u_0}{\partial x_j}(x) & \\
 r_0^n(x,y)=A^n(y)\nabla_x u_0+A^n(y)\nabla_y
w_1^n(x,y) &
\end{array}
\end{equation}
\begin{equation}
\label{I-3.5''} \psi_\epsilon^n(x)=u_\epsilon^n(x)-u_0(x)-\epsilon
w_1^n(x,\frac{x}{\epsilon})-\epsilon^2 u_2^n(x,\frac{x}{\epsilon})
\end{equation}
\begin{equation}
\label{I-3.5'''} \xi_\epsilon^n(x)=A^n(\frac{x}{\epsilon})\nabla
u_\epsilon^n-r_0^n(x,\frac{x}{\epsilon})-\epsilon
v_*^n(x,\frac{x}{\epsilon})-\epsilon^2 v_2^n(x,\frac{x}{\epsilon})
\end{equation}
 Note that
\begin{equation}
\label{I-3.6} A^n(\frac{x}{\epsilon})\nabla
\psi_\epsilon^n(x)-\xi_\epsilon^n(x)=\epsilon^2(
v_2^n(x,\frac{x}{\epsilon})-A^n(\frac{x}{\epsilon})\nabla_x
u_2^n(x,\frac{x}{\epsilon}))
\end{equation}
We have
\begin{lemma}
\label{I-t3.3}
$$(i)\;\;||\psi_\epsilon^n||_{W^{1,1}(\Omega)}<C\;\mbox{ and }||\xi_\epsilon^n||_{L^1(\Omega)}<C$$
and there exists $\psi_\epsilon\in W^{1,1}(\Omega)$ and
$\xi_\epsilon\in L^1(\Omega)$ such that
$$\psi_\epsilon^n\stackrel{n}{\rightharpoonup}
\psi_\epsilon\;,\;\nabla\psi_\epsilon^n\stackrel{n}{\rightharpoonup}
\nabla\psi_\epsilon\;,\;\xi_\epsilon^n\stackrel{n}{\rightharpoonup}
\xi_\epsilon\;,\;\mbox { weakly-* in the sense of measures}.$$ Also
we have
$$\psi_\epsilon(x)=u_\epsilon(x)-u_0(x)-\epsilon
w_1(x,\frac{x}{\epsilon})-\epsilon^2 u_2(x,\frac{x}{\epsilon})$$
$$\xi_\epsilon(x)=A(\frac{x}{\epsilon})\nabla
u_\epsilon-r_0(x,\frac{x}{\epsilon})-\epsilon
v_*(x,\frac{x}{\epsilon})-\epsilon^2 v_2(x,\frac{x}{\epsilon})$$
 (ii) Moreover, $\xi_\epsilon\in L^2(\Omega)$, $\psi_\epsilon\in
 H^1(\Omega)$ and we have
\begin{equation}
\label{I-4.1'} A(\frac{x}{\epsilon})\nabla
\psi_\epsilon(x)-\xi_\epsilon(x)=\epsilon^2(
v_2(x,\frac{x}{\epsilon})-A(\frac{x}{\epsilon})\nabla_x
u_2(x,\frac{x}{\epsilon}))
\end{equation}
with \begin{equation}
  \label{I-4.1'''}
 \nabla\cdot\xi_\epsilon(x)=0
 \end{equation}
  in the sense of distributions.
\end{lemma}
\begin{proof}
Using the fact that, for any $i,j\in\{1,2,3\}$, $\chi_j^n,
\chi_{ij}^n\in W_{per}(Y)$ and $\psi_{ij}^n\in[W_{per}(Y)]^3$ are
bounded functions in this spaces, from the definition one can
immediately see that
$$||\psi_\epsilon^n||_{W^{1,1}(\Omega)}<C\;\mbox{ and
}||\xi_\epsilon^n||_{L^1(\Omega)}<C.$$  Recall that
 $$\displaystyle\chi_j^n\rightharpoonup \chi_j\;,\;
 \displaystyle\chi_{ij}^n\rightharpoonup\chi_{ij}\;\mbox{ in }W_{per}(Y)\;\mbox{ and }
 \psi_{ij}^n\rightharpoonup \psi_{ij}\;\mbox{ in } [W_{per}(Y)]^3.$$
 Using the above convergence results and the Appendix the statement (i)
in Lemma \ref{I-t3.3} follows immediately. Observe that
$\chi_j,\chi_{ij}\in W_{per}^{1,p}(Y)$, with $p>3$ imply
\begin{equation}
\label{I-4.3'} \psi_\epsilon\in H^1(\Omega)
\end{equation}
To prove (\ref{I-4.3'}) it is enough to see that
\begin{eqnarray} ||u_2(\cdot,\frac{\cdot}{\epsilon})||_{H^1(\Omega)}&\leq&
\epsilon^2||\chi_{ij}||_{L^\infty(Y)}||u_0||_{H^2(\Omega)}+\nonumber\\
&+& \epsilon
||\chi_{ij}||_{W^{1,p}(Y)}||u_0||_{H^3(\Omega)}+
\epsilon^2||\chi_{ij}||_{L^\infty(Y)}||u_0||_{H^3(\Omega)}\nonumber
\end{eqnarray}
the rest of the necessary estimates being trivial. Similarly, from
the definition of $r_0$, $v_*$ and $v_2$ and the hypothesis
$\chi_j,\chi_{ij}\in W_{per}^{1,p}(Y)$, with $p>3$ we see that
$\xi_\epsilon\in L^2(\Omega)$.
 Next note that we immediately have
 \begin{equation}
 \label{I-ajut}\displaystyle
 A^n(\frac{x}{\epsilon})\nabla\psi^n_\epsilon\stackrel{n}{\rightharpoonup}
 A(\frac{x}{\epsilon})\nabla\psi_\epsilon\;\mbox{ weakly-* in the
 sense of measures }.\end{equation}
 Relation (\ref{I-4.1'}) follows
immediately from (\ref{I-3.6}), (\ref{I-ajut}) the relations
(\ref{I-app2}) of Section \ref{smoothing} in the Appendix and a
limit argument based on the convergence results obtained at (i).
Recall that in the smooth case it is known from \cite{MV} that
        $$\nabla\cdot \xi_\epsilon^n=0$$
This is equivalent to
$$\displaystyle\int_{\Omega}\xi_\epsilon^n\nabla\Phi(x)dx=0\;\mbox{
for any }\Phi\in{\cal{D}}(\Omega)$$ Using the fact that
$\xi_\epsilon\in L^2(\Omega)$, and that we have
$$\xi_\epsilon^n\stackrel{n}{\rightharpoonup} \xi_\epsilon\;\mbox {
weakly-* in the sense of measures}$$ we obtain (\ref{I-4.1'''}).
We make the remark that a different proof for (\ref{I-4.1'''}) can
be found in \cite{SV} \end{proof}
 We can make the observation
 that $\chi_j,\chi_{ij}\in W_{per}^{1,p}(Y)$, with $p>3$, implies $\psi_{ij}\in W_{per}^{1,p}(Y)$. Using this we obtain,
 \begin{eqnarray}
 \label{I-4.2}
 ||\nabla_x u_2(x,\frac{x}{\epsilon})||_{L^2(\Omega)}&\leq&
 ||\chi_{ij}||_{L^\infty(Y)}||\nabla_x \frac{\partial ^2 u_0}{\partial
x_j\partial x_i}||_{L^2(\Omega)}\leq\nonumber\\
&\leq&
||\chi_{ij}||_{W^{1,p}(Y)}||u_0||_{H^3(\Omega)}\leq\nonumber\\
 &\leq& C||u_0||_{H^3(\Omega)}
\end{eqnarray}
\begin{eqnarray}
\label{I-4.3} ||v_2(x,\frac{x}{\epsilon})||_{L^2(\Omega)}&\leq&
C||\psi_{ij}||_{L^\infty(Y)}||\nabla_x \frac{\partial ^2
u_0}{\partial x_i\partial
x_j}||_{L^2(\Omega)}\leq\nonumber\\&\leq& C
\displaystyle\sum_{i,j}||\psi_{ij}||_{W^{1,p}(Y)}||u_0||_{H^3(\Omega)}\leq\nonumber\\
&\leq& C||u_0||_{H^3(\Omega)}
\end{eqnarray}
where in (\ref{I-4.3}) above we used (\ref{I-4star}). Similarly as
in \cite{MV} using (\ref{I-4.2}), (\ref{I-4.3}) in (\ref{I-4.1''}),
we arrive at
$$||A(\frac{x}{\epsilon})\nabla
\psi_\epsilon(x)-\xi_\epsilon(x)||_{L^2(\Omega)}\leq
C\epsilon^2||u_0||_{H^3(\Omega)}
$$
Consider the second boundary layer $\varphi_\epsilon$ defined as
solution of
\begin{equation}
\label{I-3.9} \nabla\cdot (A(\frac{x}{\epsilon})\nabla
\varphi_\epsilon)=0\;\mbox{ in }\; \Omega\;,\;
\varphi_\epsilon=u_2(x,\frac{x}{\epsilon})\;\mbox{ on }\;
\partial\Omega
\end{equation}
Using (\ref{I-4.3'}) and similar arguments as in \cite{MV} we obtain
that
\begin{eqnarray}
\label{I-4.4} &||u_\epsilon(x)-u_0(x)-\epsilon
w_1(x,\frac{x}{\epsilon})+\epsilon\theta_\epsilon(x) -\epsilon^2
u_2(x,\frac{x}{\epsilon})+\epsilon^2\varphi_\epsilon||_{H^1_0(\Omega)}\leq\nonumber\\
&\leq C\epsilon^2 ||u_0||_{H^3(\Omega)}
\end{eqnarray}

Next we make the observation that without any further regularity
assumption on $u_0$ or on the matrix of coefficients $A$  one
cannot make use of neither Avellaneda compactness result nor the
maximum principle to obtain a $L^2$ or $H^1$ bound for
$\varphi_\epsilon$ . In fact in \cite{A} it is presented an
example where a solution of (\ref{I-3.9}) would blow up in the
$L^2$ norm. Although the unboundedness of $\varphi_\epsilon$ in
$L^2$ we can still make the observation that using a result due to
Luc Tartar \cite{T} (see also \cite{DC}, Section 8.5) concerning
the limit analysis of the classical homogenization problem in the
case of weakly convergent data in $H^{-1}(\Omega)$ together with a
few elementary computations we can obtain that
$$\epsilon\varphi_\epsilon\stackrel{\epsilon}{\rightharpoonup} 0\;\mbox{ in }
H^1(\Omega)$$ Then applying Proposition \ref{I-lemma-5.2} with
$h=0$, $y_\epsilon=\epsilon\varphi_\epsilon$,
$\phi_*(y)=\chi_{ij}(y)$, $\displaystyle
z_\epsilon(x)=z(x)=\frac{\partial^2 u_0}{\partial x_i\partial
x_j}$ we obtain that
\begin{equation}
\label{I-5.38}
\displaystyle||\epsilon\varphi_\epsilon||_{H^1(\Omega)}\leq
C\epsilon^{\frac{1}{2}}||u_0||_{H^3(\Omega)}
\end{equation}
Using (\ref{I-5.38}) in (\ref{I-4.4}) we have
\begin{eqnarray}
\label{I-5.38*} &\displaystyle||u_\epsilon(x)-u_0(x)-\epsilon
w_1(x,\frac{x}{\epsilon})+\epsilon\theta_\epsilon(x)-\epsilon^2\chi_{ij}(\frac{x}{\epsilon})\frac{\partial^2
u_0}{\partial x_i\partial x_j}||_{H^1(\Omega)}\leq\nonumber\\
&\leq C\epsilon^{\frac{3}{2}}||u_0||_{H^3(\Omega)}
\end{eqnarray}
and this concludes the proof of Theorem \ref{I-t3.2}
\end{proof}
\begin{rem}
Following similar arguments as in the above Theorem, we can adapt
the result in Theorem 3.2 (the convex case)  in \cite{G1} and
obtain
\begin{equation}
\displaystyle||\epsilon\varphi_\epsilon||_{L^2(\Omega)}\leq
C\epsilon||u_0||_{H^3(\Omega)}
\end{equation}
\end{rem}
\begin{rem}
\label{I-5.39}
 It has been shown in \cite{SSB} that the assumptions $\chi_j,\chi_{ij}\in W_{per}^{1,p}(Y)$ for some $p>N$
 are implied by the conditions that the BMO semi-norm  norm of the coefficients matrix $a$
 is small enough (see \cite{SSB} for the precise statement). In a different work by M. Vogelius and Y.Y. Lin \cite{VL},
  it has been shown
 that one can have $\chi_j,\chi_{ij}\in W_{per}^{1,\infty}(Y)$ in
 the case of piecewise discontinuous matrix of coefficients when the
 discontinuities occur on certain smooth interfaces (see \cite{VL} for
 the precise statement). It is
 clear that the lack of smoothness in the matrix $A$ and the fact
 that we only assume $u_0\in H^3(\Omega)$ would not allow one to use neither Avellaneda
 compactness principle nor the maximum principle to obtain bounds for
 $\varphi_\epsilon$ in $L^2$ or $H^1$.
\end{rem}
\begin{rem}
\label{I-Meyers} For $N=2$ we could use a Meyers type regularity
result and prove that there exists $p>2$ such that
$\chi_j,\chi_{ij}\in W_{per}^{1,p}(Y)$. Therefore Theorem
\ref{I-t3.2} holds true in this case in the very general conditions
that $u_0\in H^3(\Omega)$ and $A\in L^\infty(Y)$.
\end{rem}

\section{A natural extra term in the first order corrector to the
homogenized eigenvalue of a periodic composite medium}

 In this section we analyze the Dirichlet eigenvalues of an elliptic
 operator corresponding to a composite medium with periodic
 microstructure. This problem was initially studied in \cite{MV},
 for the case of $C^\infty$ coefficients. We generalize their result
 to the case of $L^\infty$ coefficients.

  \noindent We will first state a simple consequence of Theorem \ref{I-t3.2} which
 will play a fundamental role further in our analysis.
 \begin{cor}
 \label{I-c6.1}
 Let $\Omega\subset{\mathbb R}^2$ be a bounded, convex curvilinear polygon of class $C^\infty$.
  Let $A\in L^\infty(Y)$ and $u_0\in H^{2+r}(\Omega)$ with $r>0$. Then there exists a constant $C_r$ independent of $u_0$ and
 $\epsilon$ such that
$$||u_\epsilon(.)-u_0(.)-\epsilon
w_1(.,\frac{.}{\epsilon})+\epsilon\theta_\epsilon(.)
||_{L^2(\Omega)}\leq C_r\epsilon^{1+r}||u_0||_{H^{2+r}(\Omega)}$$
\end{cor}
\begin{proof}
From  Theorem 2.1 in \cite{OV}, if $u_0\in H^2(\Omega)$, we have
\begin{equation}
\label{I-6.0} \displaystyle||u_\epsilon(\cdot)-u_0(\cdot)-\epsilon
w_1(\cdot,\frac{.}{\epsilon})+\epsilon\theta_\epsilon(\cdot)||_{H_0^1(\Omega)}\leq
C\epsilon||u_0||_{H^2(\Omega)}
\end{equation}
Indeed note that using Remark \ref{I-Meyers} and the properties of
$Q_\epsilon$ we have that
\begin{equation}
\label{OV-2.1} ||\nabla w_1(\cdot,\frac{.}{\epsilon})-\nabla
u_1(\cdot,\frac{.}{\epsilon})||_{L^2(\Omega)} <\epsilon
||\chi_j||_{W^{1,p}(Y)}||u_0||_{H^2(\Omega)}<C
\end{equation}
where $w_1$ and $u_1$ are defined at (\ref{I-first-cor}) and
(\ref{I-first-ord-bl-reg}), respectively. Also, using the
definition of $\theta_\epsilon$ and $\beta_\epsilon$, we have
$$||\nabla \theta_\epsilon-\nabla\beta_\epsilon||_{L^2(\Omega)}<
||w_1(\cdot,\frac{.}{\epsilon})-
u_1(\cdot,\frac{.}{\epsilon})||_{H^1(\Omega)}<C$$ Using the last
two inequalities in (\ref{OV-2.1}) we obtain (\ref{I-6.0}). Next
we may see that, for $u_0\in H^3(\Omega)$, Theorem \ref{I-t3.2}
immediately implies that
\begin{equation}
\label{I-6.0'} ||u_\epsilon(.)-u_0(.)-\epsilon
w_1(.,\frac{.}{\epsilon})+\epsilon\theta_\epsilon(.)
||_{L^2(\Omega)}\leq C\epsilon^2||u_0||_{H^{3}(\Omega)}
\end{equation}
From (\ref{I-6.0}) and (\ref{I-6.0'}) together with an
interpolation argument (see Theorem 2.4 in \cite{MV}), we prove
the statement of the Corollary. \end{proof} Next we will state the
spectral problem and recall briefly the result obtained in
\cite{MV}.
 On the domain $\Omega\subset{\mathbb R}^2$, we consider the
 spectral problem (\ref{I-spectral-epsilon}) associated with operator
 $L_\epsilon$, i.e.,
\begin{equation}
\label{I-6.1}
 \left\{\begin{array}{ll}
 L_\epsilon v_\epsilon=-\nabla\cdot(A(\displaystyle\frac{x}{\epsilon})\nabla
 v_\epsilon(x))=\lambda^\epsilon v_\epsilon
 & \mbox{ in }\Omega \\
 v_\epsilon = 0 & \mbox{ on }\partial\Omega
 \end{array}\right .\end{equation}
If we consider the eigenvalue problem for the operator
$L\doteq-div({\cal{A}}^{hom}\nabla)$ with ${\cal{A}}^{hom}$
defined at (\ref{I-hom-problem}), i.e.,
\begin{equation}
\label{I-6.2}
 \left\{\begin{array}{ll}
 Lv=\lambda v
 & \mbox{ in }\Omega \\
 v = 0 & \mbox{ on }\partial\Omega
 \end{array}\right .\end{equation}
then it is well known that for $\lambda$ simple eigenvalue of
(\ref{I-6.2}), for each $\epsilon$ small enough, there exists
$\lambda^\epsilon$, an eigenvalue of (\ref{I-6.1}) such that
$$\lambda^\epsilon\stackrel{\epsilon}{\rightarrow} \lambda$$
For any $f\in L^2(\Omega)$, we define $T_\epsilon f=u_\epsilon$
where $u_\epsilon\in H_0^1(\Omega)$ is the solution of $L_\epsilon
u_\epsilon=f\;\mbox{ in }\Omega$, and similarly $Tf=u_0$ with
$u_0\in H_0^1(\Omega)$ solution of $L u_0=f$. $T_\epsilon$ and $T$
are compact and self adjoint operators from $L^2(\Omega)$ into
$L^2(\Omega)$. Moreover
$T_\epsilon\stackrel{\epsilon}{\rightarrow} T$ pointwise.

 \noindent It can be seen that $\displaystyle
\mu_k^\epsilon=\frac{1}{\lambda_k^\epsilon}$ are the eigenvalues
of $T_\epsilon$ and $\displaystyle \mu_k=\frac{1}{\lambda_k}$ are
the eigenvalues of $T$. From the definition of $T_\epsilon$ and
$T$, the eigenvectors corresponding to $\mu_k^\epsilon$ and
respectively $\mu_k$ are the same as the eigenvectors of
$L_\epsilon$ and $L$ corresponding to $\lambda_k^\epsilon$ and
respectively $\lambda_k$.

 \noindent It is proved in \cite{MV} that if $\Omega$ is a bounded convex
domain or bounded with a $C^{2,\beta}$ boundary we have that
\begin{equation}
\label{I-6.3} |\lambda-\lambda_k^\epsilon|\leq C\epsilon
\end{equation}
for $\epsilon$ sufficiently small. Moreover in the case of a
smooth matrix of coefficients $a$, and for the eigenvectors of $L$
in $H^{2+r}(\Omega)$, for some $r>0$, using (\ref{I-1.3}) and
(\ref{I-1.4}) and a result of Osborne \cite{Os}, they obtain that
$$\displaystyle\lambda^{\epsilon_n}-\lambda=
\epsilon_n\lambda\int_{\Omega}{\bar{\theta}}_{\epsilon_n}vdx+O(\epsilon_n^{1+r})$$
for any sequence $\epsilon_n\rightarrow 0$ and
${\bar{\theta}}_\epsilon$ defined by
\begin{equation}
\label{I-6.4} \displaystyle-\nabla\cdot(A(\frac{x}{\epsilon})\nabla
{\bar{\theta}}_\epsilon)=0\;\mbox{ in
}\;\Omega\;\;,\;\;{\bar{\theta}}_\epsilon=\chi_j(\frac{x}{\epsilon})\frac{\partial
v}{\partial x_j}\;\mbox{ on }\;\partial\Omega
\end{equation}

From the Corollary \ref{I-c6.1} we obtain that the result of
Moskow and Vogelius (see Theorem 3.6) remains true in the general
case of nonsmooth coefficients, i.e.,
\begin{thm}
\label{I-thm-5.2}
 In the hypothesis of Corrolary \ref{I-c6.1} if $\lambda_*$ is the
 limit of the sequence\vspace{0.1cm}
 $\displaystyle{\frac{(\lambda^{\epsilon_n}-\lambda)}{\epsilon_n}}$
 (as
 $\epsilon_n\rightarrow 0$)
 then there exists a function $\theta_*$,
 weak limit point of the sequence ${\bar{\theta}}_\epsilon$ in
 $L^2(\Omega)$, so that
$$\lambda_*=\lambda\int_{\Omega}\theta_*vdx$$
 Conversely, if $\theta_*$ is a weak limit point of the sequence ${\bar{\theta}}_\epsilon$ in
 $L^2(\Omega)$ the there exists a sequence $\epsilon_n\rightarrow 0$
 such that
$$\displaystyle
 {\frac{(\lambda^{\epsilon_n}-\lambda)}{\epsilon_n}}\rightarrow\lambda\int_{\Omega}\theta_*vdx$$.
\end{thm}
In the end we make the observation that the case when $\lambda$ is
a multiple eigenvalue can be treated similarly as in \cite{MV}
(see Remark 3.7).
\bigskip
\appendix

{\bf\large{Appendices}}

\bigskip

In the following appendices we will present the proofs for some of the results
used in the previous Sections not included in the
main body of the chapter for the sake of clarity of the exposition.
\section{Definition and Properties of the Unfolding Operator}

Let $\Xi_\epsilon\!=\!\{\xi\in{{\mathbb Z}}^N;(\epsilon\xi+\epsilon
Y)\cap\Omega\neq\emptyset\}$ and
 define
\begin{equation}
\label{I-extended-domain}
 {\tilde{\Omega}}_\epsilon=\displaystyle\bigcup_{\xi\in\Xi_\epsilon}(\epsilon\xi+\epsilon
 Y)
 \end{equation}
Let us also consider $H^1_{per}(Y)$ to be the closure of
$C^\infty_{per}(Y)$ in the $H^1$ norm, where $C^\infty_{per}(Y)$
is the subset of $C^\infty({\mathbb R}^N)$ of $Y$-periodic
functions, and
$$W_{per}(Y)\doteq \left\{v\in H^1_{per}(Y)/{{\mathbb R}}\;,\;\frac{1}{|Y|}\int_{Y}v
dy=0\right\}$$ (see \cite{DC} for properties).

 \noindent Next, similarly as in \cite{CDG},\cite{DA1}, if we have a periodical
net on ${\mathbb R}^N$ with period $Y$, by analogy with the
one-dimensional case, to each $x\in{\mathbb R}^N$ we can associate
its integer part, $[x]_Y$, such that $x-[x]_Y\in Y$ and its
fractional part respectively, i.e, $\{x\}_Y=x-[x]_Y$. Therefore we
have:
$$x=\epsilon\left\{\frac{x}{\epsilon}\right\}_Y+\epsilon\left[\frac{x}{\epsilon}\right]_Y\;\mbox{
for any }x\in {\mathbb R}^N.$$ We will recall in the following the
definition of  the Unfolding Operator as it have been introduced
in \cite{CDG}(see also \cite{DA1}), and review a few of its
principal properties. Let the unfolding operator be defined as
${\cal{T}}_\epsilon:L^2({\tilde{\Omega}}_\epsilon)\rightarrow
L^{2}({\tilde{\Omega}}_\epsilon\times Y)$ with
$${\cal{T}}_\epsilon(\phi)(x,y)=\phi(\epsilon\left
[\frac{x}{\epsilon}\right ]_{Y}+\epsilon y)\;\mbox{ for all
}\;\phi\in L^2({\tilde{\Omega}}_\epsilon)$$ We have (see
\cite{CDG}):
\begin{thm}
\label{I-app-t1}
 For any $v,w\in L^2(\Omega)$ we have
$$1.\;{\cal{T}}_\epsilon(vw)={\cal{T}}_\epsilon(v){\cal{T}}_\epsilon(w)$$
$$2.\;\nabla_{y}\left ({\cal{T}}_\epsilon(u)\right )=\epsilon
{\cal{T}}_\epsilon(\nabla_{x}u)\;\mbox{ where }\;u\in
H^1(\Omega)$$
$$3.\;\displaystyle\int_{\Omega}u dx =
\displaystyle\frac{1}{|Y|}\int_{{\tilde{\Omega}}_{\epsilon}\times
Y}{\cal T}_{\epsilon}(u)dxdy $$
$$4.\; \left
|\displaystyle\int_{\Omega}udx-\displaystyle\int_{\Omega\times
Y}{\cal T}_{\epsilon}(u)dxdy\right|<\displaystyle
|u|_{L^1(\{x\in{\tilde{\Omega}}_{\epsilon};\;dist(x,\partial\Omega)<{\sqrt
n}\epsilon\})}$$
$$5.\;{\cal{T}}_\epsilon(\psi)\rightarrow \psi\;\mbox{ uniformly on
}\Omega\times Y\;\mbox{for any }\;\psi\in{\cal{D}}(\Omega)$$
$$6.\;{\cal{T}}_\epsilon(w)\rightarrow w\;\mbox{ strongly in }\;
L^{2}(\Omega\times Y)$$ 7. Let $\{w_\epsilon\}\subset
L^{2}(\Omega\times Y)$ such that $w_\epsilon\rightarrow w \mbox{
in } L^{2}(\Omega)$. Then
   $${\cal{T}}_\epsilon(w_\epsilon)\rightarrow w \mbox{ in }
L^{2}(\Omega\times
   Y)$$
8. Let $w_\epsilon \rightharpoonup w$ in $H^1(\Omega)$. Then there
exists a subsequence and $\hat{w}\in L^{2}\left (\Omega;
H_{per}^1(Y)\right )$ such that:
$$a) \;{\cal{T}}_\epsilon(w_\epsilon)\rightharpoonup w\mbox{ in
}L^{2}(\Omega; H^1(Y))$$
$$b)\; {\cal{T}}_\epsilon(\nabla
w_\epsilon)\rightharpoonup \nabla_{x}w+\nabla_{y}\hat{w}\mbox{ in }
L^{2}(\Omega\times Y)$$
\end{thm}
Another important property of the Unfolding Operator it is
presented in the next Theorem due to Damlamian and Griso, see
\cite{G}.
\begin{thm}
\label{I-per-defect} For any $w\in H^1(\Omega)$ there exists
${\hat{w}}_\epsilon\in L^2(\Omega,H^1_{per}(Y))$ such that
\begin{equation}\label{I-app0}\left\{\begin{array}{ll}
||{\hat{w}}_\epsilon||_{L^2(\Omega,H^1_{per}(Y))}\leq C ||\nabla_x
w||_{[L^2(\Omega)]^N} & \\
||{\cal{T}}_\epsilon(\nabla_x w)-\nabla_x
w-\nabla_y{\hat{w}}_\epsilon||_{L^2(Y,H^{-1}(\Omega))}\leq C\epsilon
||\nabla_x w||_{[L^2(\Omega)]^N} & \end{array}\right.
\end{equation}
where $C$ only depends on $N$ and $\Omega$.
\end{thm}
Next present some interesting technical results obtained in
\cite{G} which are used in Section 4. Define $\displaystyle
\rho_\epsilon(.)=inf\{\frac{\rho(.)}{\epsilon},1\}$ where
$\rho(x)=dist(x,\partial\Omega)$. Define also
${\hat{\Omega}}_\epsilon=\{x\in\Omega\;;\;\rho(x)<\epsilon\}$ and
for any $\phi\in L^2(\Omega)$ consider
$M_Y^\epsilon(\phi)(x)=\displaystyle\frac{1}{|Y|}\int_{Y}{\cal{T}}_\epsilon(\phi)(x,y)dy$.
Let $v\in H^2(\Omega)$ be arbitrarily fixed, and the
regularization $Q_\epsilon$ defined at (\ref{I-1.2}). Then (see
Griso \cite{G}, for the proofs)
\begin{prop} We have
\label{I-app1} $$1.\;\; \displaystyle
||\nabla_x\rho_\epsilon||_{L^\infty(\Omega)}=
||\nabla_x\rho_\epsilon||_{L^\infty({\hat{\Omega}}_\epsilon)}=\epsilon^{-1}$$
$$2.\;\; \displaystyle
||(1-\rho_\epsilon)v||_{[L^2(\Omega)]^N}\leq ||
v||_{[L^2({\hat{\Omega}}_\epsilon)]^N}\leq C\epsilon^\frac{1}{2}||v||_{H^1(\Omega)}\;\mbox{ for any }\;v\in H^1(\Omega)$$
$$3.\;\;\displaystyle \!\!\!||\nabla_x
v||_{L^2({\hat{\Omega}}_\epsilon)}\leq
C\epsilon^{\frac{1}{2}}||v||_{H^2(\Omega)}\Rightarrow
||Q_\epsilon(\nabla_x
v)||_{L^2({\hat{\Omega}}_\epsilon)}+||M_Y^\epsilon(\nabla_x
v)||_{L^2({\hat{\Omega}}_\epsilon)}\displaystyle\leq
C\epsilon^{\frac{1}{2}}||v||_{H^2(\Omega)}$$ for any $v\in
H^2(\Omega)$.
$$4.\; \displaystyle
||\psi(\frac{.}{\epsilon})||_{L^2({\hat{\Omega}}_\epsilon)}+
||\nabla_y\psi(\frac{.}{\epsilon})||_{L^2({\hat{\Omega}}_\epsilon)}\leq
C\epsilon^{\frac{1}{2}}||\psi||_{H^1(Y)}\mbox{ for every }\psi\in
H^1_{per}(Y)$$
$$5. \;\;||M_Y^\epsilon(v)||_{L^2(\Omega)}\leq
||v||_{L^2({\tilde{\Omega}}_\epsilon)}\;\mbox{ for any }\;v\in
L^2({\tilde{\Omega}}_\epsilon)$$
\begin{equation}\nonumber
 \displaystyle 6.\;\left\{ \begin{array}{ll}||v-M_Y^\epsilon(v)||_{L^2(\Omega)}\leq
C\epsilon ||\nabla
v||_{[{L^2(\Omega)}]^N} & \\
\displaystyle||v- {\cal{T}}_\epsilon(v)||_{L^2(\Omega\times Y)}\leq
C\epsilon||\nabla v||_{[{L^2(\Omega)}]^N} & \\
\displaystyle||Q_\epsilon(v)- M_Y^\epsilon(v)||_{L^2(\Omega)}\leq
C\epsilon ||\nabla v||_{[L^2(\Omega)]^N} &   {\;\mbox{ for any
}\;v\in H^1(\Omega)}\end{array}\right.\end{equation}
$$7.\;\; \displaystyle||Q_\epsilon(v)\psi(\frac{.}{\epsilon})||_{L^2(\Omega)}\leq
C||v||_{L^2({\tilde{\Omega}}_{\epsilon,2})}||\psi||_{L^2(Y)}
\;\mbox{ for any }\;v\in L^2({\tilde{\Omega}}_{\epsilon,2})\mbox{
and }\psi\in L^2(Y)$$
\end{prop}

\section{Convergence results and the smoothing argument}
\label{smoothing}

Let $m_n\in C^\infty$ be the standard mollifying sequence, i.e.,
$0<m_n\leq 1$, $\int_{{\mathbb R}^N}m_ndz=1$, $sppt(m_n)\subset
B(0,\frac{1}{n})$. Define $A^n(y)=(m_n*A)(y)$, where $a$ has been
defined in the Introduction (see (\ref{I-epsilon-problem})). We
have:
$$ 1. A^n-Y \;\mbox{ periodic matrix } $$
$$2. |A^n|_{L^\infty}<|A|_{L^\infty}$$
\begin{equation}
\label{I-app2}
 3. A^n\rightarrow A\;\mbox{ in }\;L^p\;\mbox{ for any
 }\;p\in(1,\infty)
 \end{equation}
From (\ref{I-app2}) we have that $c|\xi|^2\leq
A^n_{ij}(y)\xi_i\xi_j\leq C|\xi|^2\;\forall \xi\in{{\mathbb
R}}^N$. Define
\begin{equation}
\label{I-aa}
 ({\cal{A}}^{hom}_n)_{ij}=\displaystyle
M_Y(A^n_{ij}(y)+A^n_{ik}(y)\frac{\partial\chi^n_j}{\partial y_k})
\end{equation} where $M_Y(\cdot)=\displaystyle\frac{1}{|Y|}\int_Y\cdot dy$
and $\chi^n_j\in W_{per}(Y)$ are the solutions of the local problem
\begin{equation}
\label{I-first-local}
 -\nabla_y\cdot(A(y)(\nabla\chi^n_j+e_j))=0
\end{equation}
Next we present a few important convergence results needed in the
smoothing argument developed in the previous Sections.
\begin{lemma}
\label{I-app-l1} Let $f_n, f\in H^{-1}(\Omega)$ with
$f_n\rightharpoonup f$ in $H^{-1}(\Omega)$ and let $b^n, b\in
L^{\infty}(\Omega)$, with
$$c|\xi|^2\leq b^n_{ij}(y)\xi_i\xi_j\leq C|\xi|^2$$
$$c|\xi|^2\leq b_{ij}(y)\xi_i\xi_j\leq C|\xi|^2$$
for all $\xi\in{{\mathbb R}}^N$ and
$$b^n\rightarrow b\;\mbox{ in }\; L^2(\Omega)$$
Consider $ \zeta_n\in H_0^1(\Omega)$ the solution of
$$\displaystyle\int_{\Omega} b^n(x)\nabla \zeta_n\nabla\psi
dx=\int_{\Omega} f_n\psi dx$$ for any $\psi\in H_0^1(\Omega)$. Then
we have
$$\zeta_n\rightharpoonup \zeta\;\mbox{ in }\;H^1_0(\Omega)$$
and $\zeta$ verifies
$$\displaystyle\int_{\Omega} b(x)\nabla
\zeta\nabla\psi dx=\int_{\Omega} f\psi dx\;\mbox{ for any }\;\psi\in
H_0^1(\Omega).$$
\end{lemma}
\begin{proof}
Immediately can be observed that
$$
||\zeta_n||_{H_0^1(\Omega)}\leq C$$ and therefore there exists
$\zeta$ such that on a subsequence still denoted by $n$ we have
\begin{equation}
\label{I-app3} \zeta_n\rightharpoonup \zeta\;\mbox{ in
}\;H_0^1(\Omega)
\end{equation}
For any smooth $\psi\in H_0^1(\Omega)$ easily it can be seen that
$$\displaystyle\int_{\Omega} b^n(x)\nabla \zeta_n\nabla\psi
dx\rightarrow \displaystyle\int_{\Omega} b(x)\nabla
\zeta\nabla\psi dx$$ and this implies the statement of the Lemma.
Due to the uniqueness of $\varphi$ one can see that the limit
(\ref{I-app3}) holds on the entire sequence. \end{proof}
\begin{rem}
\label{I-app-r0} Using similar arguments it can be proved that the
results of Lemma \ref{I-app-l1} hold true if we replace the
Dirichlet boundary conditions with periodic boundary conditions.
\end{rem}
\begin{cor}
\label{I-app-r1} Let $u_\epsilon^n\in H_0^1(\Omega)$ be the solution
of
\begin{equation}
\nonumber
 \left\{\begin{array}{ll}
 -\nabla\cdot(A^n(\displaystyle\frac{x}{\epsilon})\nabla u^n_\epsilon)=f
 & \mbox{ in }\Omega \\
 u^n_\epsilon = 0 & \mbox{ on }\partial\Omega
 \end{array}\right .\end{equation}
We then have
$$u_\epsilon^n\stackrel{n}{\rightharpoonup} u_\epsilon\;\mbox{ in
}\;H_0^1(\Omega)$$ where $u_\epsilon$ verifies
\begin{equation}
\nonumber
 \left\{\begin{array}{ll}
 -\nabla\cdot(A(\displaystyle\frac{x}{\epsilon})\nabla u_\epsilon)=f
 & \mbox{ in }\Omega \\
 u_\epsilon = 0 & \mbox{ on }\partial\Omega
 \end{array}\right .\end{equation}
\end{cor}
\begin{proof}
Using (\ref{I-app2}) we have that $$\displaystyle
A^n(\frac{x}{\epsilon})\stackrel{n}\rightarrow \displaystyle
A(\frac{x}{\epsilon})\;\mbox{ in }\;L^2(\Omega)$$ and the
statement follows immediately from Remark \ref{I-app-r0}.
\end{proof}
\begin{cor}
\label{I-app-r2} for $j\in\{1,...,N\}$, let $\chi_j^n\in W_{per}(Y)$
be the solution of
\begin{equation}
\label{I-2star'}
 \displaystyle-\nabla_y\cdot(A^n(y)(\nabla
 \chi_j^n+e_j))=0
\end{equation}
where $\{e_j\}_j$ denotes the canonical basis of ${\mathbb R}^N$.
Then we have
$$\chi_j^n\rightharpoonup \chi_j\;\mbox{ in }\; W_{per}(Y)$$
where $\chi_j\in W_{per}(Y)$ verifies
\begin{equation}
\nonumber
 \displaystyle-\nabla_y\cdot(A(y)(\nabla
 \chi_j+e_j))=0
\end{equation}
\end{cor}
\begin{proof}
From (\ref{I-app2}) we obtain
$$\displaystyle\frac{\partial}{\partial
y_i}A_{ij}^n(y)\rightharpoonup \frac{\partial}{\partial
y_i}A_{ij}(y)\;\mbox{ in }\; (W_{per}(Y))'$$ The statement of the
Remark follows then immediately from Remark \ref{I-app-r0}.
\end{proof}
\begin{prop}
\label{I-app-r3} Let $v\in [H^1(\Omega)]^N$ be arbitrarily fixed
and for every $j\in\{1,..,N\}$, let $\chi_j\in W_{per}(Y)$ be
defined as in (\ref{I-first-local}), and $\chi_j^n\in W_{per}(Y)$,
for $j\in\{1,..,N\}$, to be the solutions of (\ref{I-2star'}).

 \noindent Define $\displaystyle
h^n(x,\frac{x}{\epsilon})=\chi_j^n(\frac{x}{\epsilon}) v_j$,
$\displaystyle h(x,\frac{x}{\epsilon})=\chi_j(\frac{x}{\epsilon})
v_j$, $\displaystyle
g^n(x,\frac{x}{\epsilon})=\chi_j^n(\frac{x}{\epsilon})
Q_\epsilon(v_j)$, $\displaystyle
g(x,\frac{x}{\epsilon})=\chi_j(\frac{x}{\epsilon})
Q_\epsilon(v_j)$. We have that
$$1.\;\;g^n\stackrel{n}{\rightharpoonup}g\;\mbox{ in }\; H^1(\Omega)$$
$$2.\;\;\mbox{ If } v\in [W^{1,p}(\Omega)]^N,\;p>N,\mbox{ then }\;, h^n\stackrel{n}{\rightharpoonup}h\;\mbox{ in }\; H^1(\Omega)$$
\end{prop}
\begin{proof}
First note that applying Corollary \ref{I-app-r2} to the sequence
$\{\chi_j^n\}_{n}$ we have
\begin{equation}
\label{I-app4} \chi_j^n\stackrel{n}{\rightharpoonup} \chi_j\;\mbox{
in }\; W_{per}(Y)
\end{equation}
Next we have
\begin{eqnarray}
 \label{I-b}
 \displaystyle||g^n(x,\frac{x}{\epsilon})||^2_{H^1(\Omega)}&=&\int_{\Omega}
 (\chi_j^n(\frac{x}{\epsilon})Q_\epsilon(v_j))^2dx+
 \displaystyle\frac{1}{\epsilon^2}\int_{\Omega}(\nabla_y\chi_j^n(\frac{x}{\epsilon})Q_\epsilon(v_j))^2dx+\nonumber\\
&+&\displaystyle\int_{\Omega}(\chi_j^n(\frac{x}{\epsilon})\nabla_x
Q_\epsilon(v_j))^2dx
 \end{eqnarray}
 and
\begin{eqnarray}
 \label{I-a}
 \displaystyle||h^n(x,\frac{x}{\epsilon})||^2_{H^1(\Omega)}&=&\int_{\Omega}
 (\chi_j^n(\frac{x}{\epsilon})v_j)^2dx+
 \displaystyle\frac{1}{\epsilon^2}\int_{\Omega}(\nabla_y\chi_j^n(\frac{x}{\epsilon})v_j)^2dx+\nonumber\\
&+&\displaystyle\int_{\Omega}(\chi_j^n(\frac{x}{\epsilon})\nabla_x
v_j)^2dx
 \end{eqnarray}
For the first convergence in Theorem \ref{I-app-r3} we use that
\begin{equation}
\label{I-b'}
 ||\chi_j^n(\frac{x}{\epsilon})Q_\epsilon(v_j)||_{H^1(\Omega)}\leq
 C||\chi_j^n||_{W_{per}(Y)}
 \end{equation}
We can see that (\ref{I-b}) imply that
$$\displaystyle
||g^n(x,\frac{x}{\epsilon})-g(x,\frac{x}{\epsilon})||_{L^2(\Omega)}^2=
\displaystyle\int_{\Omega}\left(\chi_j^n(\frac{x}{\epsilon})-\chi_j(\frac{x}{\epsilon})\right)^2
(Q_\epsilon(v_j))^2dx$$
 and using (\ref{I-b'}) we obtain the desired result.

  \noindent For the second convergence result in Theorem \ref{I-app-r3} we
will recall now a very important inequality (see \cite{Lady}, Chp.
 2) to be used for our estimates. For any $p>N$ we have
\begin{equation}
 \label{I-app5}
\displaystyle||\phi||_{L^{\frac{2p}{p-2}}(\Omega)}\leq
c(p)(||\phi||_{L^2(\Omega)}+||\nabla\phi||^{\frac{N}{p}}_{L^2(\Omega)}||\phi||_{L^2(\Omega)}^{1-\frac{N}{p}})
\end{equation}
for any $\phi\in H^1(\Omega)$ and where $c(p)$ is a constant which
depends only on $q,N,\Omega$. Then, for $v\in [W^{1,p}(\Omega)]^N$
with $p>N$, using (\ref{I-app4}), the Sobolev embedding
$W^{1,p}(\Omega)\subset L^\infty(\Omega)$ and (\ref{I-app5}) in
(\ref{I-a}) we obtain
 $$\displaystyle||h^n(x,\frac{x}{\epsilon})||^2_{H^1(\Omega)}<C$$
 where the constant $C$ above does not depend on $n$.

 \noindent Next we can easily observe that
$$\displaystyle
||h^n(x,\frac{x}{\epsilon})-h(x,\frac{x}{\epsilon})||_{L^2(\Omega)}^2=
\displaystyle\int_{\Omega}\left(\chi_j^n(\frac{x}{\epsilon})-\chi_j(\frac{x}{\epsilon})\right)^2
(v_j)^2dx$$ and in either of the above cases, (\ref{I-app4}) and a
few simple manipulations imply that
$$h^n(x,\frac{x}{\epsilon})\stackrel{n}{\rightarrow}
h(x,\frac{x}{\epsilon})\;\mbox{ in }\;L^2(\Omega)$$ This together
with the bound on the sequence $\{h^n(x,\frac{x}{\epsilon})\}_n$
implies the statement of the Corollary. \end{proof} The two
convergence results in the next Corollary will follow immediately
from Proposition \ref{I-app-r2}.
\begin{cor}
\label{I-u1-w1} Let $\displaystyle
w_1^n(x,\frac{x}{\epsilon})=\chi_j^n(\frac{x}{\epsilon})\frac{\partial
u_0}{\partial x_j}$ and $\displaystyle
u_1^n(x,\frac{x}{\epsilon})=\chi_j^n(\frac{x}{\epsilon})Q_\epsilon(\frac{\partial
u_0}{\partial x_j})$. Then we have

1. If $u_0\in W^{3,p}(\Omega)$ for $p>N$,
$$w_1^n\stackrel{n}{\rightharpoonup}w_1\;\mbox{ in }\; H^1(\Omega)$$
2. If $u_0\in H^2(\Omega)$,
$$u_1^n\stackrel{n}{\rightharpoonup}u_1\;\mbox{ in }\; H^1(\Omega)$$
\end{cor}
\begin{cor}
\label{I-u1-w1-1cor} Let $\theta_\epsilon^n$ be the solution of
\begin{equation}
\label{I-w1-1cor-n} \displaystyle
-\nabla\cdot(A^n(\frac{x}{\epsilon})\nabla \theta_\epsilon^n)=
0\;\mbox{ in
}\Omega\;,\;\theta_\epsilon^n=w_1^n(x,\frac{x}{\epsilon})\;\mbox{ on
}\partial\Omega
\end{equation}
and $\beta_\epsilon^n$ be the solution of
\begin{equation}
\label{I-u1-1cor-n} \displaystyle
-\nabla\cdot(A^n(\frac{x}{\epsilon})\nabla \beta_\epsilon^n)=
0\;\mbox{ in
}\Omega\;,\;\beta_\epsilon^n=u_1^n(x,\frac{x}{\epsilon})\;\mbox{ on
}\partial\Omega
\end{equation}
We have that

(i) if $u_0\in W^{3,p}(\Omega)$, $p>N$, then
$$\theta_\epsilon^n\stackrel{n}{\rightharpoonup}
\theta_\epsilon\;\mbox{ in }\;H^1(\Omega)$$ \;\;\;\;(ii) if
$u_0\in H^2(\Omega)$, then
$$\beta_\epsilon^n\stackrel{n}{\rightharpoonup}
\beta_\epsilon\;\mbox{ in }\;H^1(\Omega)$$ where $\theta_\epsilon$
and $\beta_\epsilon$ satisfies
\begin{equation}
\label{I-w1-1cor} \displaystyle
-\nabla\cdot(A(\frac{x}{\epsilon})\nabla \theta_\epsilon)= 0\;\mbox{
in }\Omega\;,\;\theta_\epsilon=w_1(x,\frac{x}{\epsilon})\;\mbox{ on
}\partial\Omega
\end{equation}
and
\begin{equation}
\label{I-u1-1cor} \displaystyle
-\nabla\cdot(A(\frac{x}{\epsilon})\nabla \beta_\epsilon)= 0\;\mbox{
in }\Omega\;,\;\beta_\epsilon=u_1(x,\frac{x}{\epsilon})\;\mbox{ on
}\partial\Omega
\end{equation}
\end{cor}
\begin{proof}
Using Corollary \ref{I-u1-w1} and a few simple arguments one can
simply show that
        $$\displaystyle -\nabla\cdot(A(\frac{x}{\epsilon})\nabla
w_1^n(x,\frac{x}{\epsilon}))
        \stackrel{n}{\rightharpoonup}-\nabla\cdot(A(\frac{x}{\epsilon})\nabla
        w_1(x,\frac{x}{\epsilon}))\;\mbox{ in }\;H^{-1}(\Omega)$$
        and
       $$\displaystyle -\nabla\cdot(A(\frac{x}{\epsilon})\nabla
w_1^n(x,\frac{x}{\epsilon}))
        \stackrel{n}{\rightharpoonup}-\nabla\cdot(A(\frac{x}{\epsilon})\nabla
        w_1(x,\frac{x}{\epsilon}))\;\mbox{ in }\;H^{-1}(\Omega)$$
Homogenizing the data in the problems (\ref{I-w1-1cor-n}) and
(\ref{I-u1-1cor-n}) and using Corollary \ref{I-u1-w1} and Lemma
\ref{I-app-l1} the statement follows immediately.\end{proof}
\begin{cor}
\label{I-app-r4} For any $i,j\in\{1,..,N\}$ let $\chi_{ij}^n\in
W_{per}(Y)$ be the solutions of:
\begin{equation}
\label{I-app6} \nabla_y\cdot(A^n\nabla_y
\chi_{ij}^n)=b^n_{ij}-M_Y(b^n_{ij})
\end{equation}
where
$$\displaystyle b^n_{ij}=-A^n_{ij}-A^n_{ik}\frac{\partial\chi_j^n}{\partial
y_k}-\frac{\partial}{\partial y_k}(A^n_{ik}\chi_j^n)$$ and $M_Y(.)$
is the average on $Y$.

 \noindent Then we have
$$\chi_{ij}^n\rightharpoonup \chi_{ij}\;\mbox{ in }\; W_{per}(Y)\;\mbox{ for any }\;i,j\in\{1,..,N\}$$
where $\chi_{ij}$ satisfies
\begin{equation}
\label{I-app7} \displaystyle\int_Y A(y)\nabla_y\chi_{ij}\nabla_y\psi
dy=(b_{ij}-M_Y(b_{ij}),\psi)_{((W_{per}(Y))', W_{per}(Y))}
\end{equation}
 for any $\psi\in W_{per}(Y)$ and with
$$b_{ij}=-A_{ij}-A_{ik}\frac{\partial\chi_j}{\partial
y_k}-\frac{\partial}{\partial y_k}(A_{ik}\chi_j).$$
\end{cor}
\begin{proof}
 For any $\psi\in W_{per}(Y)$, we have that,
\begin{eqnarray}
\label{I-imp}
 \displaystyle\int_{Y}(b_{ij}^n-M_Y(b_{ij}^n)\psi dy&=&
\displaystyle\int_{Y}(-A^n_{ij}-A^n_{ik}\frac{\partial\chi_j^n}{\partial
y_k})\psi dy+ ({\cal{A}}^{hom}_n)_{ij}\displaystyle\int_Y\psi
dy+\nonumber\\
&+&\displaystyle\int_{Y}A^n_{ki}\chi_j^n\frac{\partial\psi}{\partial
y_k}dy
\end{eqnarray}
 where we have used that $M_Y(b_{ij}^n)=-({\cal{A}}^{hom}_n)_{ij}$ (see
\cite{MV}).

 \noindent Using (\ref{I-app2}), (\ref{I-app4}), and simple manipulations we
can prove that
\begin{equation}
\label{I-aaa} A_{ik}^n\frac{\partial\chi_j^n}{\partial
y_k}\rightharpoonup A_{ik}\frac{\partial\chi_j}{\partial
y_k}\;\mbox{ in }\;L^2(Y)
\end{equation}
and
\begin{equation}
\label{I-aaaa} A_{ik}^n\chi_j^n\rightharpoonup A_{ik}\chi_j\;\mbox{
in }\;L^2(Y)
\end{equation}
From (\ref{I-aaa}), (\ref{I-app2}) and (\ref{I-aa})  we have that
\begin{equation}
\label{I-aaaaa} ({\cal{A}}^{hom}_n)_{ij}\rightarrow
{\cal{A}}^{hom}_{ij}
\end{equation}
Finally using (\ref{I-app2}), (\ref{I-app4}), (\ref{I-aaa}) and
 (\ref{I-aaaa}) in (\ref{I-imp}) we obtain that
 $$b_{ij}^n-M_Y(b_{ij}^n)\rightharpoonup b_{ij}-M_Y(b_{ij})\;\mbox{ in
 }\; (W_{per}(Y))'$$
 This and Remark \ref{I-app-r0} complete the proof of the statement.
\end{proof}
\begin{rem}
 \label{I-rem-simpla}
 We can easily observe that we have
$$\displaystyle A_{ij}^n\chi_{ij}^n\stackrel{n}{\rightharpoonup}A_{ij}\chi_{ij},
 \;A_{ij}^n\frac{\partial\chi_{ij}^n}{\partial y_k}\stackrel{n}{\rightharpoonup}A_{ij}\frac{\partial\chi_{ij}}{\partial
 y_k}\;\mbox{ weakly in } W_{per}(Y)$$
 \end{rem}

\end{document}